\documentclass{amsart}
\usepackage{amsmath,amscd}
\usepackage{amsthm}
\usepackage{amsfonts}
\usepackage{amssymb}
\usepackage{tikz}
\usepackage{tikz-cd}
\usepackage{mathtools}
\usetikzlibrary{positioning}
\usepackage{graphicx,rotating}
\usepackage{esint}
\usepackage{float}
\usepackage{hyperref}
\usepackage{enumerate}
\usepackage{blkarray}
\usepackage{MnSymbol}

\title{$n$-Kazhdan groups and higher spectral expanders}
\author{Arghya Mondal}
\address{Chennai Mathematical Institute, H1, SIPCOT IT Park, Siruseri, Kelambakkam 603103,
India}
\email{amondal@cmi.ac.in}
\email{mondalarghya1990@gmail.com}
\urladdr{https://sites.google.com/view/arghyamondal/home}
\date{\today}



\newtheorem{theorem}{Theorem}[section]
\newtheorem{lemma}[theorem]{Lemma}
\newtheorem{prop}[theorem]{Proposition}
\newtheorem{cor}[theorem]{Corollary}

\newtheorem{question}[theorem]{Question}

\newenvironment{definition}[1][Definition]{\begin{trivlist}
\item[\hskip \labelsep {\bfseries #1}]}{\end{trivlist}}

\def\r{\mathbb{R}}
\def\c{\mathbb{C}}

\def\z{\mathbb{Z}}
\def\G{\Gamma}
\def\L{\Lambda}
\def\g{\gamma}
\def\b{\backslash}

\begin{document}
\maketitle

\begin{abstract}
    Let $\G$ be a group of type $F_n$ and let $X$ be the $n$ skeleton of the universal cover of a $K(\G,1)$ simplicial complex with finite $n$ skeleton. We show that if $\G$ is strongly $n$-Kazhdan, then for any family of finite index subgroups $\{\L_i\}_i$, the family of simplicial complexes $\{\L_i\b X\}_i$ are bounded degree $n$-dimensional spectral expanders. Using this we construct new examples of $2$ dimensional spectral expanders.
\end{abstract}

\section{Introduction}
The aim of this article is to answer the following question:
\begin{question}\label{qn}
 Is there a higher dimensional analog of the construction of expander graphs using Cayley graphs of discrete groups having Property (T)?
\end{question}
The construction we are referring to is originally due to Margulis, who used relative Property (T) of the pair $(\text{SL}_2(\z)\ltimes \z^2,\z^2)$ \cite{margulis}, later generalized by Alon and Milman \cite{am}. They considered a discrete group $\G$ with Property (T) whose Cayley graph is $X$ and showed that the family of quotients graphs $\{\L_i\b X\}_i$, where $\{\L_i\}_i$ are finite index subgroups of $\G$, are expanders. We replace expander graphs by \textit{higher dimensional spectral expanders}. Following \cite[Definition 2.3]{lubotzky18}, they are a family of $d$-dimensional bounded degree (each vertex is contained in a bounded number of cells) finite simplicial complexes, with number of vertices going to infinity, such that the \textit{upper Laplacians}, acting on the $i^\text{th}$ simplicial cochain groups with coefficients in $\c$, $0<i<d-1$, have \textit{spectral gaps} bounded below by a positive constant. (See \S\ref{sec-isosha} for definition of upper Laplacian and spectral gap.) The higher analog of having Property (T), equivalently, being \textit{$1$-Kazhdan}, is being \textit{$n$-Kazhdan}. Following \cite{dglt}, a group $\G$ is said to be \textit{$n$-Kazhdan} (resp. \textit{strongly $n$-Kazhdan}) if $H^n(\G,\pi)=0$ (resp. $H^i(\G,\pi)=0, 1\le i\le n$), for any unitary representation $\pi$ of $\G$. To find a higher analog of a Cayley graph, one notes that the Cayley graph of a group $\G$ is the $1$ skeleton of the universal cover of a $K(\G,1)$ CW complex having exactly one vertex. The natural $n$ dimensional analog is the $n$-skeleton of the universal cover of a $K(\G,1)$ complex. To make sure that we get a family of finite complexes we assume that $\G$ is of \textit{type $F_n$}, that is, there exists a $K(\G,1)$ CW complex, whose $n$ skeleton is finite. Note that when $n=1$, this comes for free, since any discrete group with Property (T) is finitely generated, that is, of type $F_1$. Further, in the $n=1$ case, by choosing finite index subgroups that do not contain the set of generators, we can ensure that the family of graphs have no loops, that is, they are 1 dimensional simplicial complexes. For higher dimensions, we use a classical result of Whitehead which says that any finite CW complex is homotopy equivalent to a finite simplicial complex, see \cite[Theorem 13]{whitehead} or \cite[Theorem 2C.5]{hatcher}. Thus we choose a $K(\G,1)$ complex which is a \textit{simplicial} complex with finite $n$-skeleton. We answer  Qn. \ref{qn}. in the following manner.

\begin{theorem}\label{showcase}
Let $\G$ be a group of type $F_n$ and let $X$ be the $n$-skeleton of the universal cover of a $K(\G,1)$ simplicial complex with finite $n$ skeleton. Let  $\{\L_i\}_i$ be a family of finite index subgroups of $\G$. If $\G$ is strongly $n$-Kazhdan  then the family $\{\L_i\b X\}_i$ of simplicial complexes are bounded degree spectral expanders.
\end{theorem}
The assumption of $\G$ being strongly $n$-Kazhdan is superfluous in the above statement. All we need is $H^i(\G,\hat{\oplus}_iL^2(\L_i\b\G))=0, 1\le i\le n$. (For an if and only if condition, generalizing the equivalence of Property ($\tau$) and the family of corresponding Cayley graphs being expanders \cite[Theorem 4.3.2]{lubotzky94}, see Corollary \ref{genproptau}). In particular, since $\G$ is finitely generated by assumption, there are finitely many subgroups of a fixed finite index and hence the family $\{\L_i\}_i$ is at most countable. Hence $\hat{\oplus}_iL^2(\L_i\b\G)$ is a separable Hilbert space. Thus, to apply Theorem \ref{showcase}, it is enough to find groups whose cohomology, with coefficients in \textit{separable} Hilbert spaces, vanish up to dimension $n$. The groups in the following result were shown to satisfy this condition for $n=2$ by Mok \cite[Corollary 0.2 (2)]{mok}.
\begin{cor}\label{new}
	Let $\G$ be a torsion free cocompact lattice in the isometry group of a non-Hermitian Riemannian symmetric space of non-compact type and of rank $\ge 6$. Then any family of finite $2$ dimensional simplicial complexes obtained from $\G$, as described in Theorem \ref{showcase}, are spectral expanders.  
\end{cor}
Let us explain why these examples are new, but related, both in source and in spirit of proof, to the most well know examples of higher spectral expanders - cocompact lattices in simple algebraic groups $G$ over non-archimedean local fields. These groups act on contractible simplicial complexes $X$, known as the Bruhat-Tits buildings. On the archimedean side, lattices in semisimple linear Lie groups $G$ act on contractible manifolds $X$, known as symmetric spaces. Thus $H^*(\G\b X,\r)\cong H^*(\G,\r)$. Working on the archimedean side, for each $p$, Matsushima \cite{matsushima62} gave a local criterion, involving the curvature tensor, for $H^p(\G\b X,\r)$ to be $G$-invariant. Garland \cite{garland} gave an analogous local criterion, in the non-archimedean setting, for vanishing of $H^p(\G\b X,\r)\cong H^p(\G,\r)$, which proved Serre's conjecture that $H^p(\G,\r)=0$ for $1\le p<\text{rk}(G)$, provided the residue field is large enough. As noted by Borel \cite[\S1]{borel}, the \textit{Garland method} in fact produces a spectral gap for the $p^\text{th}$ upper Laplacian for any finite simplicial complex given such a gap for Laplacians acting on vertices of links of $p-1$ simplices. Due to the local nature of Garland's criterion, if it is satisfied for a finite simplicial complex, it is also satisfied for any finite cover. Thus, the Garland method shows that quotients of Bruhat Tits building, by cocompact lattices in groups over fields with large residue fields, are spectral expanders. Garland's method was generalized by Zuk \cite{zuk}, Pansu \cite{pansu98} and Ballman-\'Swiatkowski \cite{bs}, to show vanishing of group cohomology with \textit{any} twisted coefficents, in particular the first cohomology, giving new examples of groups with Property (T). See also \cite{dj1}, \cite{dj2}, \cite{o1} and \cite{o2}. Examples of 2-Kazhdan groups obtained by the same method was used by de Chiffre, Glebsky, Lubotzky and Thom \cite{dglt} to produce examples of groups that are not Frobenius approximated. On the other hand Matsushima's result was generalized to twisted coefficients by Mok \cite{mok}, which implies Corollary \ref{new}. 

We note here that Garland's local point of view, has been adopted as definition of higher spectral expander by Oppenheim \cite{o3} and exmaples of such expanders have been constructed by Kaufman and Oppenheim \cite{ko1}, \cite{ko2}.

Although Theorem \ref{showcase} is stated in terms of $n$-Kazhdan groups, we think of this property as a combination of two properties -  $H^n(\G,\pi)=\overline{H^n}(\G,\pi)$ and $\overline{H^n}(\G,\pi)=0$, for all unitary representations $\pi$. Here $\overline{H^n}(\G,\pi)$, called the \textit{reduced cohomology of $\G$ with coefficients in $\pi$}, denotes the quotient of the space of cocycles by the \textit{closure} of the space of coboundaries, with respect to a natural topology. See \S\ref{sec-redcom} for details. If a group $\G$ satisfies only the first condition, then it is said to have \textit{Property \emph{(}$\emph{T}_{n-1}$\emph{)}}, a term coined by Bader and Nowak \cite{bn15}, \cite{bn20}. This concept has been crucially used in our proof of Theorem \ref{showcase}. Using an idea from this proof, we make a connection of Property ($\text{T}_n$) to the Fell topology, a result similar to that of  Bergeron and Clozel \cite{bc} in a different setting.

\begin{theorem}\label{iso}
Let $\G$ be a group of type $F_{n+1}$. Then, \emph{(1)} $\G$ has Property \emph{($\text{T}_{n-1}$)} and Property \emph{($\text{T}_n$)} implies \emph{(2)} the subset $\{\pi\in\widehat{\G}:H^n(\G,\pi)\ne 0\}$ is open in $\widehat{\G}$ with respect to the Fell topology.
\end{theorem}
Note that $H^0(\G,\pi)$ is always reduced, thus all groups have Property ($\text{T}_{-1}$). Also Property ($\text{T}_0$) is same as Property (T) \cite[Proposition 29]{bn15}. Thus for $n=0$, statement (1) in the above theorem is just that $\G$ has Property (T). On the other hand $H^0(\G,\pi)\ne 0$ if and only if $\pi$, being irreducible, is the trivial representation. Hence statement (2) is also Property (T), that is, for $n=0$ the implication in Theorem \ref{iso} is in fact an equivalence. 

Now we give a section wise description of the paper. In \S\ref{sec-redcom} we discuss reduced group cohomology and Hodge theory in this context. Corollary \ref{specgap} is the main result this section, it is used in proofs of both Theorem \ref{showcase} and Theorem \ref{iso}.  In \S\ref{sec-isosha} we show that the simplicial chain complex of $\L_i\b X$ can be identified with a chain complex of the group cohomology of $\G$ with twisted coefficients. In \S\ref{sec-proofs} we prove the result which implies Theorem \ref{showcase} and also Theorem \ref{iso}. 

\section{Reduced cohomology}\label{sec-redcom}
\subsection{Welldefinedness} Reduced 1-cohomology of groups has been studied in the context of Property (T), most notably by Shalom \cite{shalom}. More generally, reduced group cohomology of any degree has been studied by Blanc \cite{blanc}, Austin \cite{austin} and Bader and Nowak \cite{bn15}, \cite{bn20}. The cochain complexes that are used to define reduced group cohomology are either induced by the bar resolution \cite[Ch. I, \S5]{brown} or the resolution obtained from the simplicial chain complex of the universal cover of a $K(\G,1)$ simplicial complex. The fact that the different chain complexes, all with appropriate product topology on them, give the same reduced cohomology is well known to experts. But for completeness, we begin with a proof of this fact in an fairly general setting.  

Let $C:=\{C^l,d_l\}_l$ be a cochain complex of topological vector spaces over a topological field $k$ such that each $d_l$ is a continuous map. Since $d_l$ is continuous, $\ker d_l$ is closed. Therefore $\overline{\text{im }d_{l-1}}\subset\ker d_l$ and the quotient $\overline{H^l}(C):=\ker d_l/\overline{\text{im }d_{l-1}}$ is the \textit{$l^\emph{th}$ reduced cohomology} of $C$. Let $C':=\{C'^l,d'_l\}$ be another such cochain complex. If $f$ is a \textit{continuous} chain homomorphism from $C$ to $C'$, then it induces a map $\overline{f}_*:\overline{H^*}(C)\to\overline{H^*}(C')$. Note that if $f$ is null homotopic in the usual sense then $\overline{f}_*$ is $0$. Therefore, if $f:C\to C'$ and $g:C'\to C$ are continuous chain maps such that $fg-I$ and $gf-I$ are null homotopic in the usual sense, then $\overline{H^*}(C)\cong\overline{H^*}(C')$ as topological vector spaces. 

Now suppose $\G$ is a discrete group and $\c\G$ is its group algebra over $\c$. Let $\{\c\G[Y^l]\}_l$ be a free resolution of $\c$, where $\c\G[Y^l]$ is the free module over $\c\G$ with generating set $Y^l$. Let $(\pi,V_\pi)$ be a unitary representation of $\G$. Consider the cochain complex of topological vectors spaces $\{\text{Hom}_{\c\G}(\c\G[Y^l],V_\pi)\}_l$, where the topology on $\text{Hom}_{\c\G}(\c\G[Y^l],V_\pi)$ is induced from the isomorphism
\begin{equation}\label{top}
    \text{Hom}_{\c\G}(\c\G[Y^l],V_\pi)\cong V_\pi^{Y^l},
\end{equation}
and the product topology on $V_\pi^{Y^l}$. The reduced cohomology of this cochain complex is called the \textit{reduced cohomology of $\G$ with coefficients in $\pi$} and is denoted by $\overline{H^*}(\G,\pi)$. The fact that the coboundary maps are continuous follows from Lemma \ref{cont} below. Moreover, $\overline{H^*}(\G,\pi)$ is independent of the choice of free resolution since the chain homotopies induced by the chain homotopy equivalences between different free resolutions are continuous, again by Lemma \ref{cont}.     

\begin{lemma}\label{cont}
Let $f:\c\G[Y]\to\c\G[Z]$ be a $\c\G$ equivariant map. Let $(\pi,V_\pi)$ be a unitary representation of $\G$.
Then $f^*:\emph{Hom}_{\c\G}(\c\G[Z],V_\pi)\to \emph{Hom}_{\c\G}(\c\G[Y],V_\pi)$, with topologies as in \emph{(\ref{top})}, is continuous. 
\end{lemma}

\begin{proof}
Let $\{\phi_\alpha\}_\alpha\subset\text{Hom}_{\c\G}(\c\G[Z],V_\pi)$ be a net such that $\phi_\alpha\to\phi$. We want to show that $f^*\phi_\alpha\to f^*\phi$. Convergence with respect to product topology is point wise convergence. Therefore we have to show $f^*\phi_\alpha(y)\to f^*\phi(y)$ in $V_\pi$, for all $y\in Y$. Let $f(y)=\sum_{i=1}^n\lambda_i\g_iz_i$, where $\lambda_i\in\c, \g_i\in\G$ and $z_i\in Z$. Then the following computation implies the convergence.
\begin{align*}
    &\|(f^*\phi_\alpha-f^*\phi)(y)\|=\|(\phi_\alpha-\phi)(f(y))\|=\|(\phi_\alpha-\phi)(\sum_{i=1}^n\lambda_i\g_iz_i)\|\\
    \le~&\sum_{i=1}^n|\lambda_i|\|\pi(\g_i)(\phi_\alpha-\phi)(z_i)\|=\sum_{i=1}^n|\lambda_i|\|(\phi_\alpha-\phi)(z_i)\|.\qedhere
\end{align*}
\end{proof}


\subsection{Hodge theory}\label{abshodge} Fix a natural number $n$ and suppose we have a free resolution $\{\c\G[Y^l]\}_l$, where $Y^l$ is finite for all $l\le n$. This happens in particular when $\G$ is of type $F_n$ and $Y^l$ is the set of $l$ simplices of a $K(\G,1)$ simplicial complex $Y$ with finite $n+1$ skeleton. Then we can identify each element of $Y^l$ with a fixed lift of it in the universal cover $X$ of $Y$. Then the elements of $Y^l$ can be thought of as orbit representatives of the free $\G$ action on the set $X^l$ of $l$ simplices in $X$. Hence the simplicial chain complex of $X$ can be identified with $\{\c\G[Y^l]\}_l$. Since $Y^l$ is a finite set, the RHS of (\ref{top}) is a direct sum of finitely many Hilbert spaces and hence is a Hilbert space. Thus for $l\le n$ we have a cochain complex of Hilbert spaces and bounded maps. This is a familiar set up in $L^2$-cohomology. The abstract form in which we need this is given in \cite[\S3]{bn20}.  We will summarize some results from \cite[\S3]{bn20} that we need.

We say a cochain complex $C=\{C^l,d_l\}_l$ is a \textit{Hilbert cochain complex} if $C^l$ is a Hilbert space and $d_l$ is a bounded operator, for each $l$. In this case we can define the adjoint of $d_l$, which we denote by $\partial_{l+1}$.
\begin{center}
\begin{tikzcd}
C^{l-1} \arrow[r, shift left=.5ex, "d_{l-1}"]  & C^l \arrow[r, shift left=.5ex, "d_l"]\arrow[l, shift left=.5ex, "\partial_l"] 
 & C^{l+1} 
\arrow[l, shift left=.5ex, "\partial_{l+1}"] 
\end{tikzcd}
\end{center}
The \textit{upper, lower and full Laplacian for $C$ in degree $l$} are defined as
\begin{equation*}
    \Delta^+_l:=\partial_{l+1}d_l,~ \Delta^-_l:=d_{l-1}\partial_l,~\Delta_l:=\Delta^+_l+\Delta^-_l,
\end{equation*}
respectively. The space $C^l$ can be written as an orthogonal direct sum of the following closed subspaces.
\begin{align*}
    C^l_0:=\ker d_l\cap\ker\partial_l,~ C^l_-&:=\overline{\text{im }d_{l-1}},~ C^l_+:=\overline{\text{im }\partial_{l+1}}.\\
    C^l=C^l_-&\oplus C^l_0\oplus C^l_+.
\end{align*}
The kernel of $d_l$ is $C^l_0\oplus C^l_-$ and that of $\partial_l$ is $C^l_+\oplus C^l_0$. An immediate implication is $\overline{H^l}(C)\cong C^l_0$. The operators $d_l$ and $\partial_l$ can be written as  compositions
\begin{align*}
    d_l:C^l\twoheadrightarrow C^l_+\xrightarrow{\bar{d}_l} C^{l+1}_-\hookrightarrow C^{l+1},\\
    \partial_l:C^l\twoheadrightarrow C^l_-\xrightarrow{\bar{\partial}_l} C^{l-1}_+\hookrightarrow C^{l-1},
\end{align*}
where first maps in both compositions are orthogonal projections, the last maps are inclusions and the middle maps $\bar{d}_l,\bar{\partial}_l$ are injective with dense images.  Define
\begin{equation}\label{bar}
    \bar{\Delta}^+_l:=\bar{\partial}_{l+1}\bar{d}_l:C^l_+\to C^l_+,~ \bar{\Delta}^-_l:=\bar{d}_{l-1}\bar{\partial}_l:C^l_-\to C^l_-.
\end{equation}
The following result is crucial for us, we supply a proof for the reader's convenience.

\begin{prop}\label{redinv}\cite[Proposition 16]{bn20} Let $C$ be a Hilbert cochain complex. The following are equivalent. \emph{(1)} $H^{l+1}(C)$ is reduced. \emph{(2)} $\bar{\Delta}^+_l$ is invertible. \emph{(3)}  $\bar{\Delta}^-_{l+1}$ is invertible. 
\end{prop}
\begin{proof}
We will show (1)$\Leftrightarrow$(2). The proof of (1)$\Leftrightarrow$(3) is similar. $H^{l+1}(C)$ is reduced if and only if $d_l$ is surjective onto $\overline{\text{im }d_l}=C^{l+1}_-$, if and only if $\bar{d}_l$ is an isomorphism. Hence $\bar{\partial}_{l+1}$, which is adjoint of $\bar{d}_l$, is an isomorphism. Therefore $\bar{\Delta}^+_l=\bar{\partial}_{l+1}\bar{d}_l$ is invertible. On the other hand, suppose $\bar{\Delta}^+_l$ is invertible. Then $\bar{\partial}_{l+1}$ is surjective. Since it is already injective, therefore it is an isomorphism. Hence its adjoint $\bar{d}_l$ is also an isomorphism.
\end{proof}


\subsection{Hilbert direct sum} 
Recall that the \textit{Hilbert direct sum} $\hat{\oplus}_iH_i$ is the completion of the direct sum $\oplus_iH_i$ of a family $\{H_i\}_i$ of Hilbert spaces. If $T_i$ are bounded operators on $H_i$ such that $\sup_i\|T_i\|$ is finite, then the induced map on the dense subspace $\oplus_iH_i$ of $\hat{\oplus}_iH_i$ extends to a unique bounded operator $\oplus_i T_i$ on $\hat{\oplus}_iH_i$, called the \textit{direct sum of the operators $T_i$}.

\begin{definition}
We say a family of Hilbert cochain complexes $\{C_i=\{C^l_i,(d_i)_l\}_l\}_i$ has \textit{uniformly bounded coboundaries} if  $\sup_i\|(d_i)_l\|$ is finite for each $l$. In this case, the Hilbert cochain complex obtained by taking the Hilbert direct sum of the cochain spaces $C^l_i$ and the direct sum of the coboundary operators $(d_i)_l$ will be called the \textit{Hilbert direct sum of $C_i$} and will be denoted by $\hat{\oplus}C_i$.  
\end{definition}

We need the following fact about spectrum of direct sum of self adjoint operators.
\begin{theorem}\label{unionclosure}\cite[Theorem 2.23]{teschl}
Let $\{H_i\}_i$ be a family of Hilbert spaces. Let $T_i$ be a self adjoint operator on $H_i$, for all $i$, such that $\sup_i\|T_i\|$ is finite. Then
\begin{equation*}
    \sigma(\oplus_i T_i)=\overline{\cup_i\sigma(T_i)}.
\end{equation*}
\end{theorem}
\begin{definition}
Let $T$ be a positive operator on a Hilbert space. We say $T$ has an \textit{essential gap} if there exists $\epsilon>0$ such that $\sigma(T)\subset \{0\}\cup [\epsilon,\infty)$
\end{definition}

\begin{prop}\label{abspecgap}
Let $\{C_i\}_i$ be a family of Hilbert cochain complexes with uniformly bounded coboundaries. For any $l$, the following are equivalent.
\begin{enumerate}
    \item[\emph{(1)}] $H^{l+1}(\hat{\oplus}_iC_i)$ is reduced.
    \item[\emph{(2)}] $(\Delta^-_i)_{l+1}$, for all $i$, admit a uniform essential gap. 
    \item[\emph{(3)}]$(\Delta^+_i)_l$, for all $i$, admit a uniform essential gap.
 \end{enumerate}
\end{prop}
\begin{proof}
We will show (1)$\Leftrightarrow$(2). The argument for (1)$\Leftrightarrow$(3) is entirely analogous. Let $\Delta^-_{l+1}$ denote the lower Laplacians in degree $l+1$ of the cochain complex $\hat{\oplus}_iC_i$. Note that $\sigma(\Delta^-_{l+1})\subset\{0\}\cup\sigma(\bar{\Delta}^-_{l+1})$. Since, by definition, any operator is invertible if and only if its spectrum does not contain $0$, therefore $\bar{\Delta}^-_{l+1}$ is invertible if and only if $\Delta^-_{l+1}$ has an essential gap. By Theorem \ref{unionclosure}, this is equivalent to a uniform essential gap for $(\Delta_i^-)_{l+1}$ for all $i$. On the other hand, by Proposition \ref{redinv}, $\bar{\Delta}^-_{l+1}$ is invertible if and only if $H^{l+1}(\hat{\oplus}_iC_i)$ is reduced. This completes the proof.
\end{proof}
We will apply the above result to group cohomology. First we need a lemma.
\begin{lemma}\label{unifbddcobd}
Let $\{\c\G[Y^l]\}_{l\le n}$ be a partial free resolution of $\c$ by $\c\G$ modules, such that each $Y^l$ is finite. For each $l\le n$, there exists $M_l$ such that for any unitary representation $(\pi,V_\pi)$ of $\G$ the $l^\text{th}$ coboundary map $d^\pi_l$ of the cochain complex $\{\emph{Hom}_{\c\G}(\c\G[Y^l],V_\pi)\}_{l\le n}$ satisfies $\|d^\pi_l\|\le M_l$.
\end{lemma}
\begin{proof}
Since $Y^l$ is finite, $\text{Hom}_{\c\G}(\c\G[Y^l],V_\pi)\cong \c\G[Y^l]^*\otimes_{\c\G}V_\pi$, where $\c\G[Y^l]^*:=\text{Hom}_{\c\G}(\c\G[Y^l],\c\G)$ is the dual right $\c\G$ module. Let $\delta_j$ denote the $j^\text{th}$ boundary map for the chain complex  $\{\c\G[Y^l]\}_{l\le n}$. Let $\delta_j^*$ denote the dual $\c\G$ morphism. Then $d^\pi_l=\delta_{l+1}^*\otimes 1_\pi$, where $1_\pi$ is the identity map on $V_\pi$ (cf. Lemma \ref{homtensor}). Let $\{f^l_j:j\in I_l\}$ be the basis of $\c\G[Y^l]^*$ dual to the basis $Y^l$ of $\c\G[Y^l]$. For each $j\in I_l$, $\delta_{l+1}^*(f^l_j)=\sum_{i\in I_{l+1}}f^{l+1}_ia_{ij}$, for some $a_{ij}\in\c\G$. Any element of $\c\G[Y^l]^*\otimes_{\c\G}V_\pi$ is of the form $\sum_jf^l_j\otimes v_j$, where $v_j\in V_\pi$. In fact $\sum_jf^l_j\otimes v_j\mapsto (v_j)_j$ gives the isomorphism $\c\G[Y^l]^*\otimes_{\c\G}V_\pi\cong V_\pi^{Y^l}$. So $\|\sum_jf^l_j\otimes v_j\|^2=\sum_j\|v_j\|_\pi^2$. We have 
\begin{align*}
    &\|d^\pi_l(\sum_jf^l_j\otimes v_j)\|^2=\|(\sum_j\delta_{l+1}^*(f^l_j)\otimes v_j)\|^2=\|\sum_{i,j}f^{l+1}_ia_{ij}\otimes v_j\|^2\\
    =~&\|\sum_{i}f^{l+1}_i\otimes\sum_j\pi(a_{ij})v_j\|^2=\sum_i\|\sum_j\pi(a_{ij})v_j\|_\pi^2\le\sum_{i,j}\|\pi(a_{ij})\|^2\|v_j\|_\pi^2.
\end{align*}   
Any $a\in\c\G$ is of the form $a=\sum_{k=1}^m\lambda_k\g_k$, where $\lambda_k\in\c$ and $\g_k\in\G$. Then $\|\pi(a)\|\le\sum_{k=1}^m|\lambda_k|$, which is independent of $\pi$. Thus there exists $M^l_{ij}$ such that $\|\pi(a_{ij})\|\le M^l_{ij}$ for any $\pi$. Putting $M_l:=(\sum_{i,j}(M^l_{ij})^2)^{1/2}$, we get $\|d^\pi_l\|\le M_l$. 
\end{proof}
The following result is similar to \cite[Proposition 15]{bn15}.
\begin{cor}\label{specgap}
Let $\{\pi\}_i$ be a family of unitary representations of a group $\G$ of type $F_n$. For any $l< n$, the following are equivalent.
\begin{enumerate}
    \item[\emph{(1)}] $H^{l+1}(\G,\hat{\oplus}\pi_i)$ is reduced.
    \item[\emph{(2)}] $(\Delta^-_i)_{l+1}$, for all $i$, admit a uniform essential gap. 
    \item[\emph{(3)}]$(\Delta^+_i)_l$, for all $i$, admit a uniform essential gap.
 \end{enumerate}
\end{cor}
\begin{proof}
Let $\{\c\G[Y^l]\}_l$ be a free resolution of $\c$, with $Y^l$ finite for $l\le n$. By Lemma \ref{unifbddcobd}, the family $\{\text{Hom}_{\c\G}(\c\G[Y^l],V_{\pi_i})\}_{l\le n}$ of Hilbert cochain complexes has uniformly bounded coboundaries. Thus, to apply Proposition \ref{abspecgap}, all we need to show is that $\{\text{Hom}_{\c\G}(\c\G[Y^l],\hat{\oplus}V_{\pi_i})\}_{l\le n}$ is the Hilbert direct sum of this family. Fix $l\le n$ and let $|Y^l|$ be the cardinality of $Y^l$. Apropos of (\ref{top}), the claim follows from the following isomorphism.
\begin{equation*}
    (\hat{\oplus}_iV_{\pi_i})^{Y^l}\cong\bigoplus_{j=1}^{|Y^l|}\hat{\oplus}_iV_{\pi_i}=\hat{\oplus}_i\bigoplus_{j=1}^{|Y^l|}V_{\pi_i}\cong\hat{\oplus}_iV_{\pi_i}^{Y^l}.\qedhere
\end{equation*}
\end{proof}


\section{A version of Shapiro's Lemma}\label{sec-isosha}
Given a simplicial complex $Y$ we denote by $Y^l$ the set of $l$ simplices of $Y$. Recall that the \textit{simplicial chain complex of $Y$ with coefficients in $\c$} is the chain complex $\{\c[Y^l]\}_l$ of complex vector spaces $\c[Y^l]$ generated by $Y^l$ with a certain  boundary map. See \cite[\S2.1]{hatcher} for details. 
Dualizing the chain complex we get the \textit{simplicial cochain complex $\{C^l(Y,\c):=\emph{Hom}_\c(\c[Y^l],\c)\}_l$ of $Y$ with coefficients in $\c$}. Now suppose that $Y$ is a finite simplicial complex. Then we can put an inner product on each $\text{Hom}_\c(\c[Y^l],\c)$ which makes the basis dual to $Y^l$ orthonormal. Such an inner product was first considered by Eckmann \cite{eckmann45}. As in \S\ref{abshodge}, we can define the upper, lower and full Laplacian $\Delta^+_l,\Delta^-_l$ and $\Delta_l$ on $C^l(Y,\c)$. Let $d$ be denote the coboundary maps. 
\begin{definition}
The \textit{spectral gap} of upper Laplacian $\Delta^+_l$ is the minimum eigenvalue of $\Delta^+_l$ restricted to the subspace $(dC^l(Y,\c))^\perp$ orthogonal to $dC^l(Y,\c)$, with respect to the above inner product. 
\end{definition}

The rest of this section is devoted to proving the following result.

\begin{theorem}\label{isoch}
Let $\G$ and $X$ be as in Theorem \ref{showcase}. Let $\L$ be a subgroup of finite index in $\G$. Let $\{C^l(\L\b X,\c)\}_l$ be the simplicial cochain complex of $\L\b X$. Let $\{C^l(\G,L^2(\L\b\G))\}_{l\le n}$ be the truncated cochain complex for the group cohomology of $\G$ with coefficients in $L^2(\L\b\G)$, obtained from the partial resolution of $\c$ by the simplicial chain complex of $X$. We have the following chain isomorphism.
\begin{equation}\label{iso2}
    \{C^l(\L\b X,\c)\}_l\cong \{C^l(\G,L^2(\L\b\G))\}_{l\le n}.
\end{equation}
Moreover, both sides of (\ref{iso2}) have inner products defined as above and in \S\ref{abshodge}, with respect to which (\ref{iso2}) is an isometry. 
\end{theorem}
\begin{proof}
We will show that the two cochain complexes are images of the truncated simplicial chain complex of $X$ under two functors. To prove chain isomorphism it is then enough to show that the two functors are naturally isomorphic.
Recall from \S\ref{abshodge}, that we may write the simplicial chain complex of $X$ as $\{\c\G[Y^l]\}_l$, where $Y$ is the finite $K(\G,1)$ complex and $Y^l$ is the set of $l$ cells in $Y$. For each $l\le n$, $\c\G[Y^l]$ is a finitely generated free $\c\G$ module. The functors will be from the category $\c\G$-\textbf{mod} of left $\c\G$ modules to the category $\textbf{Vect}_\c$ of vector spaces over $\c$. We will provide a natural transformation between the two functors, which will be a natural isomorphism when restricted to the full subcategory of finitely generated free modules.

We begin with the functor which takes the chain complex $\{\c\G[Y^l]\}_{l\le n}$ to the cochain complex $\{C^l(\L\b X,\c):=\text{Hom}_\c(\c[\L\b X^l],\c)\}_l$. The covering map $\pi:X\to\L\b X$ induces a chain homomorphism  $\pi_l:\c[X]\to\c[\L\b X^l]$, which is surjective for all $l$. Thus we take the view that the boundary map for the simplicial chain complex of $\G\b X$ is induced from that of $X$. Algebraically, this is happening since boundary maps for $\{\c\G[Y^l]\}_{l\le n}$ are morphisms in $\c\G$-\textbf{mod}.
Consider the map
\begin{align}
\begin{split}\label{ann}
    (\text{mod }\L):\c\G&\to\c[\L\b\G]\\ \sum_ic_i\g_i&\mapsto\sum_ic_i\L\g_i. 
\end{split}
\end{align}
Then $I=\{a\in\c\G:(\text{mod }\L)a=0\}$ is a right ideal in $\c\G$ and $\c[\L\b X^l]=\c[\L\b\G Y^l]\cong\c\G[Y^l]/I\c\G[Y^l]$. Thus $C^l(\L\b X,\c)$ is the image of $\c\G[Y^l]$ under the composition of functors 
\begin{align}
\begin{split}\label{funct1}
        \c\G\text{-}\textbf{mod}&\to\textbf{Vect}_\c\hspace{1cm}\text{and}\hspace{1cm}\textbf{Vect}_\c\to\textbf{Vect}_\c\\
    M&\mapsto M/IM\hspace{2.5cm}V\mapsto\text{Hom}_\c(V,\c).
\end{split}
\end{align}
If $\phi:M_1\to M_2$ is a morphism in $\c\G$-\textbf{mod} then $\phi(IM_1)=I\phi(M_1)\subset IM_2$ and hence we get a morphism $\phi_I:M_1/IM_1\to M_2/IM_2$ in $\textbf{Vect}_\c$. This describes the action on morphisms of the first functor in the composition (\ref{funct1}). 

The functor that takes the chain complex $\{\c\G[Y^l]\}_{l\le n}$ to the cochain complex $\{C^l(\G,L^2(\L\b\G))=\text{Hom}_{\c\G}(\c\G[Y^l],L^2(\L\b\G))\}_{l\le n}$ is by definition
\begin{align*}
    &\c\G\text{-}\textbf{mod}\to\textbf{Vect}_\c\\
    &M\mapsto\text{Hom}_{\c\G}(M,L^2(\L\b\G)).
\end{align*}
By Lemma \ref{homtensor} below, we may replace this functor by the functor
\begin{align*}
    &\c\G\text{-}\textbf{mod}\to\textbf{Vect}_\c\\
    &M\mapsto M^*\otimes_{\c\G}L^2(\L\b\G).
\end{align*}
Any function $f\in L^2(\L\b\G)$ can be identified with the vector space homomorphism from $\c[\L\b\G]$ to $\c$ that linearly extends $f$. Thus the right $\c\G$ module $\c[\L\b\G]$ and the left $\c\G$ module $L^2(\L\b\G)$ can be thought of as dual vector spaces. Then by Lemma \ref{dual} below, the above functor can be replaced by the functor
\begin{align*}
    &\c\G\text{-}\textbf{mod}\to\textbf{Vect}_\c\\
    &M\mapsto\text{Hom}_\c(\c[\L\b\G]\otimes_{\c\G}M,\c).
\end{align*}
This functor can be written as a composition of the two functors
\begin{align}
\begin{split}\label{funct2}
        &\c\G\text{-}\textbf{mod}\to\textbf{Vect}_\c\hspace{1cm}\text{and}\hspace{1cm}\textbf{Vect}_\c\to\textbf{Vect}_\c\\
    &M\mapsto \c[\L\b\G]\otimes_{\c\G}M\hspace{2.5cm}V\mapsto\text{Hom}_\c(V,\c).
\end{split}
\end{align}
By (\ref{ann}), $\c[\L\b\G]$ is a cyclic right $\c\G$ module with cyclic vector the identity coset, which is annihilated by $I$. Now Lemma \ref{cyclic}  below tells us that the first functors of (\ref{funct1}) and (\ref{funct2}) are naturally isomorphic. This finishes the proof of chain isomorphism.

Let $Y^l=\{e^l_i:i\in I_l\}$ and let $\{f^l_i:i\in I_l\}$ be the basis of $\c\G[Y^l]^*$ dual to the basis $Y^l$ of $\c\G[Y^l]$. Let $\{\g_j:j\in J\}$ be a set of right coset representatives of $\L$ in $\G$. Let $\delta_j\in L^2(\L\b\G)$ be the delta function at $\L\g_i$. Then the natural isomorphism is sending the orthonormal basis $\{f^l_i\otimes\delta_j:i\in I_l,j\in J\}$ of $\c\G[Y^l]^*\otimes L^2(\L\b\G)$ to the orthonormal basis of $\text{Hom}_\c(\c[\L\b\G Y^l],\c)$ dual to the basis $\{\L\g_je^l_i:i\in I_l,j\in J\}$ of $\c[\L\b\G Y^l]$. This proves the isometry.
\end{proof}

\begin{lemma}\label{homtensor} 
Let $V$ be a left $\c\G$ module. Consider the two functors
\begin{align*}
        &\c\G\emph{-\textbf{mod}}\to\emph{\textbf{Vect}}_\c\hspace{1cm}\text{and}\hspace{1cm}\c\G\emph{{-}\textbf{mod}}\to\emph{\textbf{Vect}}_\c\\
    &M\mapsto\emph{Hom}_{\c\G}(M,V)\hspace{2.5cm}M\mapsto M^*\otimes_{\c\G}V
\end{align*}
There is a natural transformation from the former to the latter, which is a natural isomorphism when restricted to the full subcategory of finitely generated free modules.
\end{lemma}
This is well know, see \cite[Ch. I, Proposition 8.3(b)]{brown} for example, for the natural isomorphism. We omit the proof.
\begin{lemma}\label{dual}
Let $V$ be a right $\c\G$ module. Consider the two functors
\begin{align*}
        &\c\G\emph{-\textbf{mod}}\to\emph{\textbf{Vect}}_\c\hspace{1.5cm}\text{and}\hspace{1cm}\c\G\emph{{-}\textbf{mod}}\to\emph{\textbf{Vect}}_\c\\
    &M\mapsto M^*\otimes_{\c\G}\emph{Hom}_\c(V,\c)\hspace{2cm}M\mapsto\emph{Hom}_\c(V\otimes_{\c\G}M,\c)
\end{align*}
There is a natural transformation from the former to the latter. If $V$ is finite dimensional over $\c$, it is a natural isomorphism when restricted to the full subcategory of finitely generated free modules.
\end{lemma}
\begin{proof}
We have a $\c$ bilinear map
\begin{align}\label{pairing}
\begin{split}
  \langle ~,~\rangle:(M^*\otimes_{\c\G}\text{Hom}_\c(V,\c))&\times (V\otimes_{\c\G} M)\to \c,\\
  \langle f\otimes\psi,w\otimes m\rangle&=\psi(f(m)w),
\end{split}
\end{align}
which induces a $\c$ linear map $M^*\otimes_{\c\G}\text{Hom}_\c(V,\c)\to\text{Hom}_\c(V\otimes_{\c\G}M,\c), x\mapsto\langle x,~\rangle$. 
We will show that this map is a natural transformation. Let $\phi:M_1\to M_2$ be a morphism in $\c\G$-\textbf{mod}. Let $\phi^*:M_2^*\to M_1^*$ be the dual $\c\G$ morphism. Then the image of $\phi$ under the first functor is $\phi^*\otimes 1$. The image of $\phi$ under the second functor is the dual of $1\otimes\phi:V\otimes_{\c\G}M_1\to V\otimes_{\c\G}M_2$, which we denote by $(1\otimes\phi)^\#$. we have to show that the following diagram commutes:
\begin{center}
	\begin{tikzcd}
		M_2^*\otimes_{\c\G}\text{Hom}_\c(V,\c)\ar{r}\ar{d}[swap]{\phi^*\otimes 1} &\text{Hom}_\c(V\otimes_{\c\G}M_2,\c) \ar{d}{(1\otimes\phi)^\#}\\
		M_1^*\otimes_{\c\G}\text{Hom}_\c(V,\c)\ar{r} & \text{Hom}_\c(V\otimes_{\c\G}M_1,\c).
	\end{tikzcd}
\end{center}
We conclude this from the following calculation, where $f\in M_2^*,\psi\in\text{Hom}_\c(V,\c),w\in V$ and $m\in M_1$.
\begin{equation*}
    \langle f\otimes\psi,w\otimes\phi(m)\rangle_2=\psi(f(\phi(m))w)=\psi((\phi^*f)(m)w)=\langle\phi^*f\otimes\psi,w\otimes m\rangle_1.
\end{equation*}
Finally, we assume that $M$ is a finitely generated free $\c\G$ module and $V$ is  finite dimensional over $\c$. Then choosing  bases of $M$ and $V$ (over $\c\G$ and $\c$, respectively) and dual bases of $M^*$ and $\text{Hom}_\c(V,\c)$, we note that the pairing (\ref{pairing}) is non-degenerate in both coordinates and hence the natural transformation we defined is a natural isomorphism in this case. 
\end{proof}
\begin{lemma}\label{cyclic}
Let $V$ be a cyclic right $\c\G$ module with cyclic vector $v$. Let $I$ be the right ideal annihilating $v$.  The two functors below are naturally isomorphic.
\begin{align*}
        &\c\G\emph{-\textbf{mod}}\to\emph{\textbf{Vect}}_\c\hspace{1cm}\text{and}\hspace{1cm}\c\G\emph{{-}\textbf{mod}}\to\emph{\textbf{Vect}}_\c\\
    &M\mapsto V\otimes_{\c\G}M\hspace{3.2cm}M\mapsto M/IM
\end{align*}
\end{lemma}
\begin{proof}

Let $\pi:M\to M/IM$ be the quotient map. The map
\begin{align*}
    V\times M\to M/IM,\\
    (va,m)\to\pi(am),
\end{align*}
is well defined since if we replace $va$ by $v(a+b)$, where $b\in I$, then $\pi((a+b)m)=\pi(am)+\pi(bm)=\pi(am)$. It is also $\c\G$ bilinear and hence induces a map $V\otimes_{\c\G}M\to M/IM$. We will show that this map is a natural transformation. Let $\phi:M_1\to M_2$ be a morphism in $\c\G$-\textbf{mod}. The image $\phi_I$ of $\phi$ under the first functor is defined by the commutativity of the following diagram:
\begin{center}
	\begin{tikzcd}
		M_1\ar{r}{\phi}\ar{d}[swap]{\pi_1} & M_2\ar{d}{\pi_2}\\
		M_1/IM_1\ar{r}[swap]{\phi_I} & M_2/IM_2.
	\end{tikzcd}
\end{center}
We have to show the following diagram commutes:
\begin{center}
	\begin{tikzcd}
		V\otimes_{\c\G}M_1\ar{r}\ar{d}[swap]{1\otimes\phi} & M_1/IM_1\ar{d}{\phi_I}\\
		V\otimes_{\c\G}M_2\ar{r} & M_2/IM.
	\end{tikzcd}
\end{center}
The composition of the top horizontal and right vertical maps is given by 
\begin{equation*}
    va\otimes m\mapsto\pi_1(am)\mapsto\phi_I(\pi_1(am))=\pi_2(\phi(am)),
\end{equation*}
while that of the left vertical and bottom horizontal map is given by
\begin{equation*}
    va\otimes m\mapsto va\otimes\phi(m)\mapsto \pi_2(a\phi(m))=\pi_2(\phi(am)).
\end{equation*}
To show that this natural transformation is in fact a natural isomorphism, we produce an inverse to the map $V\otimes_{\c\G}M\to M/IM$. Consider the map $M\to V\otimes_{\c\G}M, m\mapsto v\otimes m$. If $m\in M$ and $a\in I$ then $am\mapsto v\otimes am=va\otimes m=0$. Thus we get a map $M/IM\to V\otimes_{\c\G}M, \pi(m)\mapsto v\otimes m$. We check that the composition $V\otimes_{\c\G}M\to M/IM\to V\otimes_{\c\G}M$ is identity:
\begin{equation*}
    va\otimes m\mapsto\pi(am)\mapsto v\otimes am=va\otimes m,
\end{equation*}
and the composition $M/IM\to V\otimes_{\c\G}M\to M/IM$ is identity: 
\begin{equation*}
    \pi(m)\mapsto v\otimes m\mapsto\pi(m).\qedhere
\end{equation*}
\end{proof}

\section{Proofs of Theorems \ref{showcase} and \ref{iso}}\label{sec-proofs}
Theorem \ref{showcase} follows immediate from Corrollary \ref{genproptau} below. 
\begin{theorem}\label{main}
Let $\G, X$ and $\{\L_i\}_i$ be as in Theorem \ref{showcase}. Let $\lambda_i^+$ and $\lambda_i^-$ be the first non-zero eigenvalues of the upper Laplacian $\Delta_i^+$ and the lower Laplacian $\Delta_i^-$, respectively, acting on $C^{n-1}(\L_i\b X,\c)$. 
\begin{enumerate}
    \item[\emph{(1)}] There exists $\epsilon>0$ such that $\lambda^+_i>\epsilon$, for all $i$, if and only if  $H^n(\G,\hat{\oplus}_iL^2(\L_i\b\G))$ is reduced.
    \item[\emph{(2)}] There exists $\epsilon>0$ such that $\lambda^-_i>\epsilon$, for all $i$, if and only if $H^{n-1}(\G,\hat{\oplus}_iL^2(\L_i\b\G))$ is reduced. 
\end{enumerate}
In addition, $\overline{H^{n-1}}(\G,\hat{\oplus}_iL^2(\L_i\b\G))=0$ iff $H^{n-1}(\L_i\b X,\c)=0$, for all $i$. 
\end{theorem}
\begin{cor}\label{genproptau}
Let $\G, X$ and $\L_i$ be as in Theorem \ref{showcase}. Then $\{\L_i\b X\}$ is a family of spectral expanders if and only if $H^i(\G,\hat{\oplus}_iL^2(\L_i\b\G))=0$, for all $1\le i\le n-1$, and $H^n(\G,\hat{\oplus}_iL^2(\L_i\b\G))$ is reduced.
\end{cor}
Corollary \ref{genproptau} follows immediately from Theorem \ref{main} by noting that the spectral gap and the essential gap of $\Delta^+_i$ are the same if and only if $H^i(\L_i\b X,\c)=0$.

\vspace{.1cm}

\noindent\textit{Proof of Theorem \ref{main}}. By Theorem \ref{isoch}, we may assume that $\Delta^+_i$ and $\Delta^-_i$ are the upper and lower Laplacians of $C^{n-1}(\G,L^2(\L_i\b\G))$, with $\lambda_i^+$ and $\lambda_i^-$ being their first non-zero eigenvalues, respectively.  Since $C^{n-1}(\G,L^2(\L_i\b\G))$ are finite dimensional, the spectrums of $\Delta^+_i$ and $\Delta^-_i$ are their sets of eigenvalues.  By Corollary \ref{specgap}, $H^n(\G,\hat{\oplus}_iL^2(\L_i\b\G))$ is reduced if and only if $\Delta_i^+$, for all $i$, admit a uniform essential gap, that is, there exists $\epsilon>0$ such that $\lambda^+_i>\epsilon$ for all $i$. This proves (1). Again by Corollary \ref{specgap}, $H^{n-1}(\G,\hat{\oplus}_iL^2(\L_i\b\G))$ is reduced if and only if $\Delta_i^-$, for all $i$, admit a uniform essential gap, that is, there exists $\epsilon>0$ such that $\lambda^-_i>\epsilon$ for all $i$. This proves (2). We will prove the additional statement assuming (1). The proof for that assuming (2) is similar. Let $\Delta^+$ denote the upper Laplacians of $C^n(\G,\hat{\oplus}_iL^2(\L_i\b\G))$.  Theorem \ref{unionclosure} says
\begin{equation*}
    \sigma(\Delta^+)=\overline{\cup_i\sigma(\Delta^+_i)}.
\end{equation*}
By (1) there is a uniform essential gap for $\Delta_i^+$, for all $i$. This is equivalent to $\Delta^+$ having a essential gap. In this situation $0\notin\sigma(\Delta^+)$ is equivalent to
$\overline{H^{n-1}}(\G,\hat{\oplus}_iL^2(\L_i\b\G))$ $=0$. It is also equivalent to $0\notin\sigma(\Delta^+_i)$, for all $i$, which in turn is equivalent to $H^{n-1}(\L_i\b X,\c)$ $\cong H^{n-1}(\G,L^2(\L_i\b\G))=0$, for all $i$.\qed

\vspace{.1cm}

Now we turn to the proof of Theorem \ref{iso}. Let $\G$ be of type $F_{n+1}$ and let $\{\c\G[Y]^l,\delta_l\}_{l\le n+1}$ be a partial resolution of $\c$ by finitely generated $\c\G$ modules. We begin with an explicit description of the Laplacian $\Delta_\pi$ on the (truncated) cochain complex $\{\c\G[Y^l]^*\otimes_{\c\G}V_\pi,\delta_{l+1}^*\otimes 1\}_{l\le{n+1}}$ for the group cohomology of $\G$  with coefficients in a unitary representation $(\pi,V_\pi)$.
\begin{lemma}\label{conjugate}
For all $l\le n+1$, let $\{f_j^l:j\in I_l\}$ be the basis of $\c\G[Y^l]^*$ dual to the basis $Y^l$ of $\c\G[Y^l]$. Let $A_l$ be the matrix of $\delta^*_l$ with respect to these bases. Then $(\Delta_\pi)_l$ is of the form $\Phi_l\otimes 1$, where the matrix of $\Phi_l$ with respect to the above basis is
\begin{equation*}
    D_l:=\iota(A_{l+1})^\top A_{l+1}+A_l\iota(A_l)^\top,
\end{equation*}
where $\iota$ is the anti-involution on $\c\G$ given by $\iota(\sum_iz_i\g_i)=\sum_i\overline{z}_i\g_i^{-1}$.
\end{lemma}

\begin{proof}
It is enough to show that the adjoint $\partial_{l+1}$ to the coboundary map $d_l=\delta^*_{l+1}\otimes 1$ is of the form $\phi_{l+1}\otimes 1$, where the matrix of $\phi_{l+1}$ is $\iota(A_l)^\top$. Let $\sum_kf^l_k\otimes w_k\in\c\G[Y^l]^*\otimes V_\pi$ and $\sum_if^{l-1}_i\otimes v_i\in\c\G[Y^{l-1}]\otimes V_\pi$. We have to show
\begin{equation*}
    \langle\sum_{i,j}f^l_ja_{ij}\otimes v_i,\sum_kf^l_k\otimes w_k\rangle=\langle\sum_if^{l-1}_i\otimes v_i,\sum_{k,j}f^{l-1}_j\iota(a_{jk})\otimes w_k\rangle.
\end{equation*}
The LHS simplifies as,
\begin{align*}
    &\langle\sum_{i,j}f^l_j\otimes\rho(a_{ij})v_i,\sum_kf^l_k\otimes w_k\rangle=\sum_{i,j}\langle f^l_j\otimes\rho(a_{ij})v_i,\sum_kf^l_k\otimes w_k\rangle\\
    =&\sum_{i,j}\langle f^l_i\otimes\rho(a_{ij})v_i,f^l_j\otimes w_j\rangle=\sum_{i,j}\langle\rho(a_{ij})v_i,w_j\rangle_\rho,
\end{align*}
while the RHS simplifies as
\begin{align*}
    &\langle\sum_if^{l-1}_i\otimes v_i,\sum_{k,j}f^{l-1}_j\otimes\rho(\iota(a_{jk}))w_k\rangle=\sum_{k,j}\langle\sum_if^{l-1}_i\otimes v_i,f^{l-1}_j\otimes\rho(\iota(a_{jk}))w_k\rangle\\
    =&\sum_{k,j}\langle f^{l-1}_j\otimes v_j,f^{l-1}_j\otimes\rho(\iota(a_{jk}))w_k\rangle=\sum_{k,j}\langle v_j,\rho(\iota(a_{jk}))w_k\rangle_\rho\\
    =&\sum_{k,j}\langle v_j,\rho(a_{jk})^*w_k\rangle_\rho=\sum _{k,j}\langle \rho(a_{jk})v_j,w_k\rangle_\rho.\qedhere
\end{align*}
\end{proof}

\begin{theorem}\label{minmax}\cite[Theorem 2.19]{teschl} Let $T$ be a bounded self adjoint operator on a Hilbert space $H$. Then
\begin{equation*}
    \inf\sigma(T)=\inf_{v\ne 0}\frac{\langle Tv,v\rangle}{\langle v,v\rangle}.
\end{equation*}
\end{theorem}
\noindent The quantity $\langle Tv,v\rangle/\langle v,v\rangle$ is called the \textit{Rayleigh quotient of $v$ for $T$}.

The key observation for the proof of Theorem \ref{iso} is the following result.
\begin{prop}\label{contfell}
Let $\G$ be a group of type $F_{n+1}$. For all $l\le n$, the following map is continuous at all $\pi$ satisfying $\inf\sigma((\Delta_\pi)_l)=0$.
\begin{align*}
    \widehat{\G} &\to [0,\infty)\\
    \pi &\mapsto\inf\sigma((\Delta_\pi)_l).
\end{align*}
\end{prop}
\begin{proof}
Let us find the expression for the Rayleigh quotient of a generic element $\sum_if_i^l\otimes v_i\in\c\G[Y^l]^*\otimes V_\pi$ for $(\Delta_\pi)_l$. By Lemma \ref{conjugate}, $(\Delta_\pi)_l$ is of the form $\Phi\otimes 1$, where $\Phi$ is given by a matrix $D$, whose elements we denote by $d_{ij},(i,j)\in I_l\times I_l$. Then the numerator of the Rayleigh quotient is
\begin{align}
\begin{split}\label{numerator}
    &\langle(\Delta_\pi)_l(\sum_if^l_i\otimes v_i),\sum_kf^l_k\otimes v_k\rangle=\langle\sum_{i,j}f_j\otimes\pi(d_{ij})v_i,\sum_kf^l_k\otimes v_k\rangle\\
    =~&\sum_{i,j}\langle f_j\otimes\pi(d_{ij})v_i,\sum_kf^l_k\otimes v_k\rangle
    =\sum_{i,j}\langle f_j\otimes\pi(d_{ij})v_i,f^l_j\otimes v_j\rangle=\sum_{i,j}\langle\pi(d_{ij})v_i,v_j\rangle_\pi,
\end{split}
\end{align}
while the denominator is $\sum_i\|v_i\|^2$. For a fixed $\g\in\G$, the map $(v,w)\mapsto\langle\pi(\g)v,w\rangle_\pi$ is a sesquilinear form on $V_\pi$. Hence by the polarization identity for sesquilinear forms (see for example \cite[Proposition A.59]{hall}), any matrix coefficient can be written as a linear combination of diagonal matrix coefficients, that is, functions of positive type. Thus the final expression in (\ref{numerator}) can be written as a linear combination of functions of positive type associated to $\pi$ evaluated at the finite number of elements in $\G$ which appear in $d_{ij}$, for all $i,j$. The terms $\|v_i\|^2$ in the denominator are also functions of positive type evaluated at the identity. Thus the Rayleigh quotient is of the form $F(\phi_1(\g_1),\cdots,\phi_m(\g_m))$, where $F$ is a rational function,  $\phi_1,\cdots,\phi_m$ are functions of positive type and $\g_1,\cdots,\g_m\in\G$. Note that as we vary over different  representations $\pi$ and different elements $\sum_if_i^l\otimes v_i$ in $\c\G[Y^l]\otimes_{\c\G}V_\pi$, the functions $\phi_1,\cdots,\phi_m$ vary, but not the elements $\g_1,\cdots,\g_m$ of $\G$. Let $K:=\{\g_1,\cdots,\g_m\}$.

To show continuity we will show that if $\{\pi_\alpha\}$ is a net in $\widehat{\G}$ such that $\pi_\alpha\to\pi$, with $\inf\sigma((\Delta_\pi)_l)=0$, then $\inf\sigma((\Delta_{\pi_\alpha})_l)\to 0$. Fix any $\epsilon>0$. By Theorem \ref{minmax} and the above discussion, there exists functions of positive type $\phi_1,\cdots,\phi_m$ associated to $\pi$ such that
\begin{equation}\label{ineq1}
    F(\phi_1(\g_1),\cdots,\phi_m(\g_m))<\epsilon/2.
\end{equation}
By continuity of $F$, there exists $\delta>0$ such that whenever $x_i\in\c$ satisfies $|x_i-\phi_i(\g_i)|<\delta$ for all $1\le i\le m$, $|F(x_1,\cdots,x_m)-F(\phi_1(\g_1),\cdots,\phi_m(\g_m))|<\epsilon/2$. Let $\tilde{W}(\pi,\phi_1,\cdots,\phi_m,K,\delta)\subset\widehat{\G}$ be the set of representations $\rho$ such that there exists functions of positive type $\psi_1,\cdots,\psi_m$ associated to $\rho$ with $|\phi_i(\g)-\psi_i(\g)|<\delta$ for all $1\le i\le m$ and $\g\in K$. Since $\pi$ is irreducible, by \cite[Proposition F.2.4]{BdlHV}, $\tilde{W}(\pi,\phi_1,\cdots,\phi_m,K,\delta)$ is a neighbourhood of $\pi$. Since $\pi_\alpha\to\pi$, there exists $\alpha_0$ such that for all $\beta>\alpha_0$, $\pi_\beta\in \tilde{W}(\pi,\phi_1,\cdots,\phi_m,K,\delta)$. Then there exist functions of positive type $\xi_1,\cdots,\xi_m$ associated to $\pi_\beta$ such that 
\begin{equation}\label{ineq2}
    F(\xi_1(\g_1),\cdots,\xi_m(\g_m))-F(\phi_1(\g_1),\cdots,\phi_m(\g_m))<\epsilon/2.
\end{equation}
Combining (\ref{ineq1}) and (\ref{ineq2}) we get $F(\xi_1(\g_1),\cdots,\xi_m(\g_m))<\epsilon$. Hence $\inf\sigma((\Delta_{\pi_\beta})_l)<\epsilon$, for all $\beta>\alpha_0$. This finishes the proof.
\end{proof}

\noindent\textit{Proof of Theorem \ref{iso}}.  Note that Property ($\text{T}_n$) and Property  ($\text{T}_{n-1}$) are equivalent to $(\Delta^+_\pi)_n$ and $(\Delta^-_\pi)_n$ having essential gaps, respectively, for all unitary representations $\pi$. The proof is by contradiction.  Suppose $\G$ has ($\text{T}_n$) and Property ($\text{T}_{n-1}$). All $n^\text{th}$ homology groups are reduced. Suppose $\mathcal{R}:=\{\pi\in\widehat{\G}:H^n(\G,\pi)\ne 0\}$ is not open in $\widehat{\G}$. Then there exists $\pi_0\in\mathcal{R}$ such that  for every neighbourhood $U$ of $\pi_0$ there exists  $\pi_U\in U\cap (\widehat{\G}\setminus\mathcal{R})$. Since $H^n(\G,\pi_0)\ne 0$ therefore $\inf\sigma((\Delta_{\pi_0})_n)=0$. By Proposition \ref{contfell}, $\inf\sigma((\Delta_{\pi_U})_n)\to 0$. Note that $0\notin\{\inf\sigma((\Delta_{\pi_U})_n)\}_U$. Since $\r$ is second countable, we replace the family $\{\pi_U\}_U$ by a countable sub family $\{\pi_i\}_i$  such that $\inf\sigma((\Delta_{\pi_i})_n)\to 0$. Let $\rho$ be the Hilbert direct sum of the representations $\{\pi_i\}_i$. Then by an argument similar to the proof of Corollary \ref{specgap}, $(\Delta_\rho)_n$ does not have essential gap. Hence either $(\Delta^+_\rho)_n$ or $(\Delta^-_\rho)_n$ does not have essential gap.\qed


\section*{Acknowledgements}
 I thank Anish Ghosh for his encouragement and continued engagement with this project. While I was looking for a higher dimensional analog of Cayley graph of a group, Swathi Krishna told me about Cayley complex. Considering the $n^\text{th}$ skeleton of the universal cover of a $K(\G,1)$ complex was a small step from there.

\end{document}